\documentclass[12pt]{article}

\nonstopmode

\usepackage[left=2cm,right=2cm,top=2cm,bottom=2cm]{geometry}
\usepackage[utf8]{inputenc}
\usepackage[english]{babel}
\usepackage[T1]{fontenc}
\usepackage{amsmath}
\usepackage{amsthm}
\usepackage{amsfonts}
\usepackage{amssymb}
\usepackage{color}
\usepackage{multicol}
\usepackage{times}
\usepackage{hyperref}
\hypersetup{
colorlinks=true, 
breaklinks=true, 
urlcolor= blue, 
linkcolor= blue, 
citecolor=blue, 
}
\usepackage{graphicx}
\usepackage{lscape}
\usepackage{setspace}

\newcommand{\norm}[1]{\left\lVert #1 \right\rVert}

\newtheorem{theorem}{Theorem}[section]
\newtheorem{lemma}{Lemma}[section]

\newtheorem{definition}{Definition}[section]

\title{On the Diophantine equation    $\displaystyle \sum _{k=1}^{5}F_{n_k}=2^a$ } 

\author{Pagdame TIEBEKABE}
\date{}
\begin{document}
\maketitle
\begin{abstract}
\noindent Let $(F_n)_{n\geq 0}$ be the Fibonacci sequence given by $F_0 = 0, F_1 = 1$ and $F_{n+2} = F_{n+1}+F_n$ for $n \geq 0$. In this paper, we have determined all the powers of 2 which are sums of five Fibonacci numbers with few exceptions that we characterize. We have also stated an open problem relating to the number of solutions of equations like those studied in this paper.

\end{abstract}

\section{Introduction} 
Solving Diophantine equations fascinates many mathematicians specializing in number theory. There are several methods of solving these equations. Among them, the most fruitful is surely the one which associates the linear forms of logarithms and some calculations with continuous fractions. 

We know that there are the only three ($1, 2$ and $8$) Fibonacci numbers that are powers of 2. A proof of this fact follows from Carmichael's primitive divisor theorem \cite{1}, which states that for $n>12$, the $n-$th Fibonacci number $F_n$ has at least one prime factor which is not a factor of any previous Fibonacci number (see the paper of Bilu, Hanrot and Voutier  \cite{Bil} for the most general version of the statement above). However, there are nine ($9$) powers of $2$ in the sum of two Fibonacci numbers   \cite{9}, fifteen ($15$) in the sum of three Fibonacci numbers \cite{5} and sixty ($60$) in the sum of four Fibonacci numbers \cite{Tieb}. 

In this article, we showed that there are exactly one hundred and six ($106$) powers of $2$ in the sum of five Fibonacci numbers. We solved the following exponential Diophantine 
\begin{equation}\label{eq1}
\displaystyle \sum _{k=1}^{5}F_{n_k}=2^a.
\end{equation}
Note that the solutions listed in section \ref{5} are non-trivial solutions. That is we considered the solutions for which, for all $k\in \{1, 2, 3, 4, 5\}$, $n_k\neq 0$.

\noindent This paper is subdivided as follows: In Section $2$, we introduce auxiliary results used in Sections $3$ and $4$ to prove the main theorem of this paper stated below. 

\begin{theorem}\label{theo4}$\\~$
All non-trivial solutions of the Diophantine equation (\ref{eq1}) in positive integers $n_1, n_2, n_3, n_4, n_5$ and $p$ with $n_1\geqslant n_2\geqslant\cdots \geqslant n_5$ are listed in Section \ref{5}. 
\end{theorem}
The method used to prove the theorem \ref{theo4} is a double application of Baker's method and some computations with continued fractions to reduce the brute force search range for the variables. In the following sections, we will first bound $n_1$. Then we will use the reduction methods stated in section $2$ to considerably reduce this bound. The last section contains tables of all solutions of the equation \ref{eq1}. We end this paper with an open problem.

\section{Auxiliary results}
In this section, we give some well-known definitions, proprieties, theorem and lemmas. 
\begin{definition}[Mahler measure]
 For all algebraic numbers  $\gamma$, we define its measure by the following identity :
\begin{equation*}
{\rm M}(\gamma)=|a_d|\prod\limits_{i=1}^d \max \{1,|\gamma_{i}|\},
\end{equation*}
where $\gamma_{i}$ are the roots of $f(x)=a_d\prod\limits_{i=1}^d(x-\gamma_{i})$ is the minimal polynomial of  $\gamma$.
\end{definition} 
 
\noindent Let us now define another height, deduced from the last one, called the absolute logarithmic height. 
\begin{definition}[Absolute logarithmic height] 
For a non-zero algebraic number of degree $d$ on $\mathbb{Q}$ where the minimal polynomial on  $\mathbb{Z}$ is $f(x)=a_d\prod\limits_{i=1}^d(x-\gamma_{i})$, we denote by
\begin{equation*}
h(\gamma)=\dfrac{1}{d}\left(\log|a_d|+\sum\limits_{i=1}^d \log\max\{1,|\gamma_i|\}\right) = \dfrac{1}{d} \log {\rm M}(\gamma).
\end{equation*}
the usual logarithmic absolute height of $\gamma$.
 \end{definition}
\noindent The following properties of the logarithmic height are well-known:
 \begin{itemize}
 \item[•] $h(\gamma \pm \eta)\leq h(\gamma)+h(\eta)+\log 2$;
 \item[•] $h(\gamma\eta^{\pm 1})\leq h(\gamma)+h(\eta)$;
 \item[•] $h(\gamma^k)=|k|h(\gamma)\quad k\in \mathbb{Z}$.
 \end{itemize}
The $n-$th Fibonacci number can be represented as
$$
F_n=\dfrac{\alpha^n-\beta^n}{\sqrt{5}}\quad \text{for all}\quad n\geqslant 0.
$$
where $(\alpha,\beta):=((1+\sqrt{5})/2,(1-\sqrt{5})/2)$. The following inequalities 
$$
\alpha^{n-2}\leqslant F_n \leqslant \alpha^{n-1}
$$
are well-known to hold for all $n\geqslant 1$ and can be proved by induction on $n$.
The following theorem is deduced from Corollary $2.3$ of Matveev \cite{3}.
\begin{theorem}[Matveev \cite{3}]\label{theo3}

\noindent Let $n\geq 1$ an integer. Let $\mathbb{L}$ be  a field of algebraic number of degree $D$. Let $\gamma_1$, \dots, $\gamma_t$ non-zero elements of \ $\mathbb{L}$ and let 
$b_1$, $b_2$, \dots, $b_t$ integers, 
$$
B:=\max\{|b_1|,...,|b_t|\},
$$ 

\noindent and
\begin{equation*}
\Lambda:=\gamma_1^{b_1}\cdots\gamma_t^{b_t}-1=\left(\prod\limits_{i=1}^t \gamma_i^{b_i}\right)-1.
\end{equation*}
Let $A_1$, \dots, $A_t$ reals numbers such that 
\begin{equation*}
A_j\geq \max\{Dh(\gamma_j),|\log (\gamma_j)|,0.16\}, 1\leq j\leq t.
\end{equation*}
Assume that  $\Lambda\neq 0$, So we have
\begin{equation*}
\log|\Lambda|>-3\times 30^{t+4}\times (t+1)^{5.5}\times d^2 \times A_1...A_t(1+\log D)(1+\log nB).
\end{equation*}
Further, if\  $\mathbb{L}$ is real, then 
\begin{equation*}
\log|\Lambda|>-1.4\times 30^{t+3}\times (t)^{4.5}\times d^2 \times A_1...A_t(1+\log D)(1+\log B).
\end{equation*}

\end{theorem}
\noindent The two following Lemmas are due to Dujella and Peth\H o, and to Legendre respectively.

\noindent For a real number $X$, we write $\norm{X} :=\min\{\mid X-n\mid :n\in \mathbb{Z}\}$ for the distance of $X$ to the nearest integer.

\begin{lemma}[Dujella and Peth\H o, \cite{2}]\label{lem1}

Let $M$ a positive integer, let $p/q$ the convergent of the continued fraction expansion of $\kappa$ such that $q>6M$ and let  $A$, $B$, $\mu$ real numbers such that  $A>0$ and $B>1$. Let $\varepsilon:=\norm{\mu q}-M\norm{\kappa q}$.\\ If $\varepsilon>0$ then  there is no solution of the inequality 
\begin{equation*}
0<m\kappa -n+\mu< AB^{-m}
\end{equation*}
 in integers $m$ and $n$ with 
$$
\dfrac{\log (Aq/\varepsilon)}{\log B}\leqslant m \leqslant M.
$$

\end{lemma}

\begin{lemma}[Legendre]\label{lem2}

Let $\tau$ real number such that $x$, $y$ are integers such that 
\begin{equation*}
\left| \tau - \dfrac{x}{y}\right|<\dfrac{1}{2y^{2}},
\end{equation*}
  then $\dfrac{x}{y}=\dfrac{p_{k}}{q_{k}}$  is the convergent of $\tau$. 
  
\end{lemma}
\noindent Further, 
  \begin{equation*}
\left| \tau - \dfrac{x}{y}\right|>\dfrac{1}{(q_{k+1}+2)y^{2}}.
\end{equation*}
\section{Main result}
We now prove our main result Theorem \ref{theo4}.
\begin{proof}
Assume that
\begin{equation*}
\displaystyle \sum _{k=1}^{5}F_{n_k}=2^a
\end{equation*}
holds.\\
Let us first find relation between $n_1$ and $a$. 

\noindent Combining equation(\ref{eq1}) with the well-known inequality $F_n\leqslant \alpha^{n-1}$ for all $n\geqslant 1$, one gets that 
\begin{align*}
\displaystyle \sum _{k=1}^{5}F_{n_k}=2^p\leqslant& \displaystyle \sum _{i=1}^{5} \alpha^{n_i-1}\\
<&\displaystyle \sum _{i=1}^{5} 2^{n_i-1} \quad \because \alpha<2\\
<& 2^{n_1-1}\left(1+ \displaystyle \sum _{i=2}^{5} 2^{n_i-n_1} \right)\\
\leqslant & 2^{n_1-1}\left( 1+1+  \displaystyle \sum _{i=1}^{3} 2^{-i}  \right)= 2^{n_1-1}\left(2+\displaystyle \sum _{i=1}^{3} 2^{-i} \right)\\
<& 2^{n_1+1}.
\end{align*}
Hence
$$
2^a<2^{n_1+1}\Longrightarrow a< n_1+1\Longrightarrow a\leqslant n_1.
$$
This inequality will help  us to calculate some parameters.

\noindent If $n_1 \leqslant 400$, then a brute force search with \textit{Mathematica}
in the range $1\leqslant n_5\leqslant n_4\leqslant n_3\leqslant n_2\leqslant n_1\leqslant 400$ turned up only the solutions shown in the statement of Theorem \ref{theo4}. This took few minutes. Thus, for the rest of the paper we assume that $n_1 > 400$.

\subsection{Upper bound for $(n_1-n_2)\log \alpha$ in terms of $n_1$}$\\~$
\begin{lemma}
If $(n_1, n_2, n_3, n_4, n_5, a)$ is a positive solution of (\ref{eq1}) with $n_1\geqslant n_2\geqslant n_3\geqslant n_4\geqslant n_5$, then 
$$
(n_1-n_2)\log \alpha<2.32\times 10^{12}\log n_1.
$$
\end{lemma}

\begin{proof}

Rewriting equation (\ref{eq1}), we get
$$
\dfrac{\alpha^{n_1}}{\sqrt{5}}-2^a= \dfrac{\beta^{n_1}}{\sqrt{5}}-(F_{n_{2}}+F_{n_{3}}+F_{n_{4}}+ F_{n_{5}}).
$$
Taking absolute values on the above equation, we obtain
$$
\left|\dfrac{\alpha^{n_1}}{\sqrt{5}}-2^a \right|\leqslant \left| \dfrac{\beta^{n_1}}{\sqrt{5}}\right|+ (F_{n_{2}}+F_{n_{3}}+F_{n_{4}}+F_{n_{5}})< \dfrac{|\beta|^{n_1}}{\sqrt{5}}+ (\alpha^{n_{2}}+\alpha^{n_{3}}+\alpha^{n_{4}}+\alpha^{n_{5}}),
$$
and
$$
\left|\dfrac{\alpha^{n_1}}{\sqrt{5}}-2^a \right|< \dfrac{1}{2}+ \left(\alpha^{n_{2}}+\alpha^{n_{3}}+\alpha^{n_{4}}+\alpha^{n_{5}}\right)\quad \text{where we used}\quad F_n\leqslant \alpha^{n-1}.
$$

\noindent Dividing both side of the above equation by $\alpha^{n_{1}}/\sqrt{5}$, we get 

\begin{align*}
\left|1-2^a\cdot \alpha^{-n_1}\cdot \sqrt{5}\right|<& \dfrac{\sqrt{5}}{2\alpha^{n_1}}+ \left(\alpha^{n_{2}-n_1}+\alpha^{n_{3}-n_1}+\alpha^{n_{4}-n_1}+\alpha^{n_{5}-n_1}\right)\sqrt{5}\\
<& \dfrac{\sqrt{5}}{2\alpha^{n_1}}+\dfrac{\sqrt{5}}{\alpha^{n_1-n_2}}+\dfrac{\sqrt{5}}{\alpha^{n_1-n_3}}+\dfrac{\sqrt{5}}{\alpha^{n_1-n_4}}+\dfrac{\sqrt{5}}{\alpha^{n_1-n_5}}.
\end{align*}

\noindent Taking into account the assumption $n_5\leqslant n_4\leqslant n_3\leqslant n_3\leqslant n_2\leqslant n_1$, we get

\begin{equation}\label{eq:2}
|\Lambda_1|=\left|1-2^a\cdot \alpha^{-n_1}\cdot \sqrt{5}\right|< \dfrac{11.5}{\alpha^{n_1-n_2}}, \quad \text{where}\quad \Lambda_1=1-2^a\cdot \alpha^{-n_1}\cdot \sqrt{5}.
\end{equation}

\noindent Let apply Matveev's theorem, with the following parameters $t:=3$ and 
$$
\gamma_1:=2,\quad \gamma_2:=\alpha,\quad \gamma_3:=\sqrt{5},\quad b_1:=a,\quad b_2:=-n,\quad \text{and}\quad b_3:=1.
$$

\noindent Since $\gamma_1, \gamma_2, \gamma_3\in \mathbb{K}:=\mathbb{Q}(\sqrt{5}),$ we can take $D:=2$. Before applying Matveev's theorem, we have to check the last condition: the left-hand side of (\ref{eq:2}) is not zero. Indeed, if it were zero, we would then get that $2^a\sqrt{5}=\alpha^n$. Squaring the previous relation, we get $\alpha^{2n}=5\cdot 2^{2a}=5\cdot 4^a$. This implies that $\alpha^{2n}\in \mathbb{Z}$. Which is impossible. Then $\Lambda_1\neq 0$. The logarithmic height of $\gamma_1, \gamma_2$ and $\gamma_3$ are:

\noindent $h(\gamma_1)=\log 2= 0.6931\ldots$, so we can choose $A_1:=1.4$.

\noindent $h(\gamma_2)=\dfrac{1}{2}\log \alpha=0.2406\ldots$, so we can choose $A_2:= 0.5$.

\noindent $h(\gamma_3)=\log \sqrt{5}=0.8047\ldots$, it follows that we can choose $A_3:=1.7$.

\noindent Since $a< n_1+1$, $B:=\max \{|b_1|, |b_2|, |b_3|\}=n_1.$ Matveev's result informs us that
\begin{equation}\label{eq:3}
\left|1-2^a\cdot \alpha^{n_1}\cdot \sqrt{5}\right|>\exp \left(-c_1\cdot (1+\log n)\cdot 1.4\cdot 0.5\cdot 1.7\right),
\end{equation}
where $c_1:=1.4\cdot 30^6\cdot 3^{4.5}\cdot 2^2\cdot (1+\log 2)<9.7\times 10^{11}.$

\noindent Taking $\log$ in inequality (\ref{eq:2}), we get
$$
\log |\Lambda_1|< \log (11.5)-(n_1-n_2)\log \alpha.
$$

\noindent Taking $\log$ in inequality (\ref{eq:3}), we get
$$
\log |\Lambda_1|> 2.31\times 10^{12}\log n_1.
$$

\noindent Comparing the previous two inequalities, we get
$$
(n_1-n_2)\log \alpha-\log (11.5) < 2.31\times 10^{12}\log n_1,
$$
where we used $1+\log n_1< 2\log n_1$ which holds for all $n_1\geqslant 3$. Then we have
\begin{equation}\label{eq:4}
(n_1-n_2)\log \alpha<2.32\times 10^{12}\log n_1.
\end{equation}
\end{proof}
\subsection{Upper bound for $(n_1-n_3)\log \alpha$ in terms of $n_1$}$\\~$
\begin{lemma}
If $(n_1, n_2, n_3, n_4, n_5, a)$ is a positive solution of (\ref{eq1}) with $n_1\geqslant n_2\geqslant n_3\geqslant n_4\geqslant n_5$, then 
$$
(n_1-n_3)\log \alpha< 3.29\times 10^{24}\log^2 n_1.
$$
\end{lemma}
\begin{proof}
Let us now consider a second linear form in logarithms. Rewriting equation (\ref{eq1}) as follows
$$
\dfrac{\alpha^{n_1}}{\sqrt{5}}+\dfrac{\alpha^{n_2}}{\sqrt{5}}-2^a=\dfrac{\beta^{n_1}}{\sqrt{5}}+\dfrac{\beta^{n_2}}{\sqrt{5}}-\left(F_{n_3}+F_{n_4}+F_{n_5}\right).
$$
Taking absolute values on the above equation and the fact that $\beta=(1-\sqrt{5})/2$, we get
\begin{align*}
\left| \dfrac{\alpha^{n_1}}{\sqrt{5}}\left(1+\alpha^{n_2-n_1}\right) -2^a\right|\leqslant &\dfrac{|\beta|^{n_1}+|\beta|^{n_2}}{\sqrt{5}}+F_{n_3}+F_{n_4}+F_{n_5}\\
<& \dfrac{1}{3}+ \alpha^{n_3}+\alpha^{n_4}+\alpha^{n_5}\quad \text{for all} \quad n_1\geqslant 5\quad \text{and}\quad n_2\geqslant 5.
\end{align*}
Dividing both sides of the above inequality by $\dfrac{\alpha^{n_1}}{\sqrt{5}}\left(1+\alpha^{n_2-n_1}\right)$, we obtain 
\begin{equation}\label{eq:5}
|\Lambda_2|=\left|1-2^a\cdot\alpha^{-n_1}\cdot\sqrt{5}\left(1+\alpha^{n_2-n_1}\right)^{-1}\right|<\dfrac{8}{\alpha^{n_2-n_1}}\quad\text{where}\quad \Lambda_2= 1-2^a\cdot\alpha^{-n_1}\cdot\sqrt{5}\left(1+\alpha^{n_2-n_1}\right)^{-1}.
\end{equation}
Let apply  Matveev's theorem for the second time with the following data 
$$
t:=3,\quad \gamma_1:=2,\quad \gamma_2:=\alpha,\quad \gamma_3:=\sqrt{5}\left(1+\alpha^{n_2-n_1}\right)^{-1},\quad b_1:=a,\quad b_2:=-n_1,\quad \text{and}\quad b_3:=1.
$$
Since $\gamma_1, \gamma_2, \gamma_3\in \mathbb{K}:=\mathbb{Q}(\sqrt{5}),$ we can take $D:=2$. The left hand side of (\ref{eq:5}) is not zero, otherwise, we would get the relation
\begin{equation}\label{eq:6}
2^a\sqrt{5}=\alpha^{n_1}+\alpha^{n_2}.
\end{equation}
Conjugating (\ref{eq:6}) in the field $\mathbb{Q}(\sqrt{5})$, we get
\begin{equation}\label{eq:7}
-2^a\sqrt{5}=\beta^{n_1}+\beta^{n_2}.
\end{equation}

\noindent Combining (\ref{eq:6}) and (\ref{eq:7}), we get 
$$
\alpha^{n_1}<\alpha^{n_1}+\alpha^{n_2}=|\beta^{n_1}+\beta^{n_2}|\leqslant |\beta|^{n_1}+|\beta|^{n_2}< 1
$$

\noindent which is impossible for $n_1>400$. Hence $\Lambda_2\neq 0$.
We know that, $h(\gamma_1)=\log 2$ and $h(\gamma_2)=\dfrac{1}{2}\log \alpha$.
Let us now estimate $h(\gamma_3)$ by first observing that
$$
\gamma_3=\dfrac{\sqrt{5}}{1+\alpha^{n_2-n_1}}<\sqrt{5}\quad \text{and}\quad \gamma_3^{-1}=\dfrac{1+\alpha^{n_2-n_1}}{\sqrt{5}}<\dfrac{2}{\sqrt{5}},
$$
so that $|\log\gamma_3|<1.$
Using proprieties of logarithmic height stated in Section $2$, we have
$$
h(\gamma_3)\leqslant \log \sqrt{5}+ |n_2-n_1|\left(\dfrac{\log \alpha}{2}\right)+\log 2= \log (2\sqrt{5})+ (n_1-n_2)\left(\dfrac{\log \alpha}{2}\right).
$$ 
Hence, we can take $A_3:=3+(n_1-n_2)\log \alpha>\max \{2h(\gamma_3), |\log \gamma_3|,0.16\}$. 

\noindent Matveev's theorem implies that 
$$
\exp \left(-c_2(1+\log n_1)\cdot 1.4\cdot 0.5\cdot(3+(n_1-n_2)\log \alpha)\right)
$$
where $c_2:=1.4\cdot 30^6\cdot 3^{4.5}\cdot 2^2\cdot (1+\log 2)<9.7\times 10^{11}.$

\noindent Since $(1+\log n_1)< 2\log n_1$ hold for $n_1\geqslant 3,$ from (\ref{eq:5}), we have
\begin{equation}\label{eq:8}
(n_1-n_3)\log \alpha-\log 8< 1.4\times 10^{12}\log n_1(3+(n_1-n_2)\log \alpha).
\end{equation} 

\noindent Putting relation (\ref{eq:4}) in the right-hand side of (\ref{eq:8}), we get

\begin{equation}\label{eq:9}
(n_1-n_3)\log \alpha< 3.29\times 10^{24}\log^2 n_1.
\end{equation}
\end{proof}
\subsection{Upper bound for $(n_1-n_4)\log \alpha$ in terms of $n_1$}$\\~$
\begin{lemma}
If $(n_1, n_2, n_3, n_4, n_5, a)$ is a positive solution of (\ref{eq1}) with $n_1\geqslant n_2\geqslant n_3\geqslant n_4\geqslant n_5$, then 
$$
(n_1-n_4)\log \alpha< 9.3\times 10^{36}\log^3n_1.
$$
\end{lemma}
\begin{proof}

Let us consider a third linear form in logarithms. To this end, we again rewrite (\ref{eq1}) as follows
$$
\dfrac{\alpha^{n_1}+\alpha^{n_2}+\alpha^{n_3}}{\sqrt{5}}-2^a=\dfrac{\beta^{n_1}+\beta^{n_2}+\beta^{n_3}}{\sqrt{5}}-F_{n_4}-F_{n_5}.
$$
Taking absolute values on both sides, we obtain
\begin{align*}
\left|\dfrac{\alpha^{n_1}}{\sqrt{5}}\left(1+\alpha^{n_2-n_1}+\alpha^{n_3-n_1}\right)-2^a\right|\leqslant & \dfrac{|\beta|^{n_1}+|\beta|^{n_2}+|\beta|^{n_3}}{\sqrt{5}}+F_{n_4}+F_{n_5}\\
< & \dfrac{3}{4}+\alpha^{n_4}+\alpha^{n_5}\quad \text{for all}\quad n_1>400,\quad \text{and}\quad n_2, n_3, n_4, n_5\geqslant 1.
\end{align*}
Thus we have
\begin{equation}\label{eq:10}
|\Lambda_3|=\left|1-2^a\cdot \alpha^{-n_1}\cdot\sqrt{5}\left(1+\alpha^{n_2-n_1}+\alpha^{n_3-n_1}\right)^{-1}\right|< \dfrac{4}{\alpha^{n_1-n_4}},
\end{equation}
$$
\text{where}\quad \Lambda_3=1-2^a\cdot \alpha^{-n_1}\cdot\sqrt{5}\left(1+\alpha^{n_2-n_1}+\alpha^{n_3-n_1}\right)^{-1}.
$$
\noindent In a third application of Matveev's theorem, we can take parameters
$$
t:=3,\quad \gamma_1:=2,\quad \gamma_2:=\alpha,\quad \gamma_3:=\sqrt{5}\left(1+\alpha^{n_2-n_1}+\alpha^{n_3-n_1}\right)^{-1},\quad b_1:=a,\quad b_2:=-n,\quad \text{and},\quad b_3:=1.
$$
Since $\gamma_1, \gamma_2, \gamma_3\in \mathbb{K}:=\mathbb{Q}(\sqrt{5}),$ we can take $D:=2$. The left hand side of (\ref{eq:10}) is not zero. The proof is done by contradiction. Suppose the contrary. Then
$$
2^a\sqrt{5}=\alpha^{n_1}+\alpha^{n_2}+\alpha^{n_3}.
$$
Taking the conjugate in the field $\mathbb{Q}(\sqrt{5})$, we get
$$
-2^a\sqrt{5}=\beta^{n_1}+\beta^{n_2}+\beta^{n_3},
$$
which leads to
$$
\alpha^{n_1}< \alpha^{n_1}+\alpha^{n_2}+\alpha^{n_3}=|\beta^{n_1}+\beta^{n_2}+\beta^{n_3}|\leqslant |\beta|^{n_1}+|\beta|^{n_2}+|\beta|^{n_3}< 1
$$
and leads to a contradiction since $n_1>400.$ Hence $\Lambda_3\neq 0$.

\noindent As we did before, we can take $A_1:=1.4, A_2:=0.5$ and $B:=n_1.$ We can also see that 
$$
\gamma_3=\dfrac{\sqrt{5}}{1+\alpha^{n_2-n_1}+\alpha^{n_3-n_1}}< \sqrt{5}\quad \text{and}\quad \gamma_3^{-1}=\dfrac{1+\alpha^{n_2-n_1}+\alpha^{n_3-n_1}}{\sqrt{5}}<\dfrac{3}{\sqrt{5}},
$$
so $|\log \gamma_3|<1$. Applying proprieties on logarithmic height, we estimate $h(\gamma_3)$.
Hence 
\begin{align*}
h(\gamma_3)\leqslant & \log \sqrt{5}+|n_2-n_1|\left(\dfrac{\log\alpha}{2}\right)+|n_3-n_1|\left(\dfrac{\log\alpha}{2}\right)+\log 3\\
=& \log(3\sqrt{5})+(n_1-n_2)\left(\dfrac{\log\alpha}{2}\right)+(n_1-n_3)\left(\dfrac{\log\alpha}{2}\right);
\end{align*}
so we can take
$$
A_3:=4+(n_1-n_2)\log \alpha+(n_1-n_3)\log \alpha> \max \{2h(\gamma_3), |\log \gamma_3|, 0.16\}.
$$
A lower bound on the left-hand side of (\ref{eq:10}) is
$$
\exp (-c_3\cdot(1+\log n_1)\cdot 1.4\cdot 0.5 \cdot(4+(n_1-n_2)\log \alpha+ (n_1-n_3)\log \alpha))
$$

\noindent where $c_3=1.4\cdot 30^6\cdot 3^{4.5}\cdot 2^2\cdot (1+\log 2)<9.7\times 10^{11}.$

\noindent From inequality (\ref{eq:10}), we have 
\begin{equation}\label{eq:11}
(n_1-n_4)\log \alpha< 1.4\times 10^{12}\log n_1\cdot(4+(n_1-n_2)\log \alpha+(n_1-n_3)\log \alpha).
\end{equation}
Combining equation (\ref{eq:4}) and (\ref{eq:9}) in the right-most terms of equation (\ref{eq:11}) and performing the respective calculations, we get
\begin{equation}\label{eq:12}
(n_1-n_4)\log \alpha< 9.3\times 10^{36}\log^3n_1.
\end{equation} 
\end{proof}
\subsection{Upper bound for $(n_1-n_5)\log \alpha$ in terms of $n_1$}$\\~$
\begin{lemma}
If $(n_1, n_2, n_3, n_4, n_5, a)$ is a positive solution of (\ref{eq1}) with $n_1\geqslant n_2\geqslant n_3\geqslant n_4\geqslant n_5$, then 
$$
(n_1-n_5)\log \alpha< 40.32\times 10^{48}\log^4n_1.
$$
\end{lemma}
\begin{proof}
Let us now consider a forth linear form in logarithms. Rerwriting (\ref{eq1}) once again by separating large terms and small terms, we get

$$
\dfrac{\alpha^{n_1}+\alpha^{n_2}+\alpha^{n_3}+\alpha^{n_4}}{\sqrt{5}}-2^a=\dfrac{\beta^{n_1}+\beta^{n_2}+\beta^{n_3}+\beta^{n_4}}{\sqrt{5}}-F_{n_5}.
$$
Taking absolute values on both sides, we get
\begin{align*}
\left|\dfrac{\alpha^{n_1}}{\sqrt{5}}\left(1+\alpha^{n_2-n_1}+\alpha^{n_3-n_1}+\alpha^{n_4-n_1}\right)-2^a\right|\leqslant & \dfrac{|\beta|^{n_1}+|\beta|^{n_2}+|\beta|^{n_3}+|\beta|^{n_4}}{\sqrt{5}}+F_{n_5}\\
< & \dfrac{4}{5}+\alpha^{n_5}\quad \text{for all}\quad n_1>400,\quad \text{and}\quad n_2, n_3, n_4, n_5\geqslant 1.
\end{align*}

\noindent Dividing both sides of the above relation by the fist term of the right hand side of the previous equation, we get 
\begin{equation}\label{eq:13}
|\Lambda_4|=\left|1-2^a\cdot\alpha^{-n_1}\cdot \sqrt{5}\left(1+\alpha^{n_2-n_1}+\alpha^{n_3-n_1}+\alpha^{n_4-n_1}\right)^{-1}\right|<\dfrac{2}{\alpha^{n_1-n_5}},
\end{equation}
$$
\text{where}\quad \Lambda_4=1-2^a\cdot\alpha^{-n_1}\cdot \sqrt{5}\left(1+\alpha^{n_2-n_1}+\alpha^{n_3-n_1}+\alpha^{n_4-n_1}\right)^{-1}.
$$
\noindent In the application of Matveev's theorem, we have the following parameters
$$
\gamma_1:=2,\quad \gamma_2:=\alpha,\quad \gamma_3:=\sqrt{5}\left(1+\alpha^{n_2-n_1}+\alpha^{n_3-n_1}+\alpha^{n_4-n_1}\right)^{-1},
$$
and we can also take $b_1:=a,\quad b_2:=-n$ and $b_3:=1.$ Since $\gamma_1, \gamma_2, \gamma_3\in \mathbb{K}:=\mathbb{Q}(\sqrt{5}),$ we can take $D:=2$. The left hand side of (\ref{eq:13}) is not zero. The proof is done by contradiction. Suppose the contrary. Then
$$
2^a\sqrt{5}=\alpha^{n_1}+\alpha^{n_2}+\alpha^{n_3}+\alpha^{n_4}.
$$
Conjugating the above relation in the field $\mathbb{Q}(\sqrt{5})$, we get
$$
-2^a\sqrt{5}=\beta^{n_1}+\beta^{n_2}+\beta^{n_3}+\beta^{n_4}.
$$
Combining the above two equations, we get
$$
\alpha^{n_1}< \alpha^{n_1}+\alpha^{n_2}+\alpha^{n_3}+\alpha^{n_4}=|\beta^{n_1}+\beta^{n_2}+\beta^{n_3}+\beta^{n_4}|\leqslant |\beta|^{n_1}+|\beta|^{n_2}+|\beta|^{n_3}+|\beta|^{n_4}< 1,
$$
and leads to contradiction since $n_1>400.$

\noindent As done before, here, we can take $A_1:=1.4, A_2:=0.5$ and $B:=n_1$. Let us estimate $h(\gamma_3)$.
We can see that, 
$$
\gamma_3=\dfrac{\sqrt{5}}{1+\alpha^{n_2-n_1}+\alpha^{n_3-n_1}+\alpha^{n_4-n_1}}<\sqrt{5}\quad \text{and}\quad \gamma_3^{-1}=\dfrac{1+\alpha^{n_2-n_1}+\alpha^{n_3-n_1}+\alpha^{n_4-n_1}}{\sqrt{5}}<\dfrac{4}{\sqrt{5}}.
$$
Hence $|\log \gamma_3|<1.$ Then 
\begin{align*}
h(\gamma_3)\leqslant &\log (4\sqrt{5})+|n_2-n_1|\left(\dfrac{\log\alpha}{2}\right)+|n_3-n_1|\left(\dfrac{\log\alpha}{2}\right)+|n_4-n_1|\left(\dfrac{\log\alpha}{2}\right)\\
=& \log(4\sqrt{5})+(n_1-n_2)\left(\dfrac{\log\alpha}{2}\right)+(n_1-n_3)\left(\dfrac{\log\alpha}{2}\right)+(n_1-n_4)\left(\dfrac{\log\alpha}{2}\right);
\end{align*}
so we can take
$$
A_3:= 5+(n_1-n_2)\log \alpha+(n_1-n_3)\log \alpha+ (n_1-n_4)\log \alpha.
$$
Then a lower bound on the left-hand side  of (\ref{eq:13}) is 
$$
\exp (-c_4\cdot(1+\log n_1)\cdot 1.4\cdot 0.5\cdot(5+(n_1-n_2)\log \alpha+(n_1-n_3)\log \alpha+ (n_1-n_4)\log \alpha)),
$$
where $c_4=1.4\cdot 30^6\cdot 3^{4.5}\cdot 2^2\cdot (1+\log 2)<9.7\times 10^{11}.$

\noindent So, inequality (\ref{eq:13}) yields 
\begin{equation}\label{eq:14}
(n_1-n_5)\log \alpha< 1.4\times 10^{12}\log n_1\cdot (5+(n_1-n_2)\log \alpha+(n_1-n_3)\log \alpha+ (n_1-n_4)\log \alpha).
\end{equation}
Using now (\ref{eq:4}), (\ref{eq:9}) and (\ref{eq:12}) in the right-most terms of the above inequality (\ref{eq:14}) and performing the respective calculation, we find that 
$$
(n_1-n_5)\log \alpha< 40.32\times 10^{48}\log^4n_1.
$$

\end{proof}         
\subsection{Upper bound for $n_1$}$\\~$

Let us now consider a fifth linear form in logarithms. Rerwriting (\ref{eq1}) once again by separating large terms and small terms, we get

$$
\dfrac{\alpha^{n_1}+\alpha^{n_2}+\alpha^{n_3}+\alpha^{n_4}+\alpha^{n_5}}{\sqrt{5}}-2^a=\dfrac{\beta^{n_1}+\beta^{n_2}+\beta^{n_3}+\beta^{n_4}+\beta^{n_5}}{\sqrt{5}}.
$$
Taking absolute values on both sides, we get
\begin{align*}
\left|\dfrac{\alpha^{n_1}}{\sqrt{5}}\left(1+\alpha^{n_2-n_1}+\alpha^{n_3-n_1}+\alpha^{n_4-n_1}+\alpha^{n_5-n_1}\right)-2^a\right|\leqslant & \dfrac{|\beta|^{n_1}+|\beta|^{n_2}+|\beta|^{n_3}+|\beta|^{n_4}+|\beta|^{n_5}}{\sqrt{5}}\\
< & \dfrac{6}{5}\quad \text{for all}\quad n_1>400,\quad \text{and}\quad n_2, n_3, n_4, n_5\geqslant 1.
\end{align*}

\noindent Dividing both sides of the above relation by the fist term of the right hand side of the previous equation, we get 
\begin{equation}\label{eq:15a}
|\Lambda_5|=\left|1-2^a\cdot\alpha^{-n_1}\cdot \sqrt{5}\left(1+\alpha^{n_2-n_1}+\alpha^{n_3-n_1}+\alpha^{n_4-n_1}+\alpha^{n_5-n_1}\right)^{-1}\right|<\dfrac{3}{\alpha^{n_1}},
\end{equation}
$$
\text{where}\quad \Lambda_5=1-2^a\cdot\alpha^{-n_1}\cdot \sqrt{5}\left(1+\alpha^{n_2-n_1}+\alpha^{n_3-n_1}+\alpha^{n_4-n_1}+\alpha^{n_5-n_1}\right)^{-1}.
$$
\noindent In the last application of Matveev's theorem, we have the following parameters
$$
\gamma_1:=2,\quad \gamma_2:=\alpha,\quad \gamma_3:=\sqrt{5}\left(1+\alpha^{n_2-n_1}+\alpha^{n_3-n_1}+\alpha^{n_4-n_1}+\alpha^{n_5-n_1}\right)^{-1},
$$
and we can also take $b_1:=a,\quad b_2:=-n$ and $b_3:=1.$ Since $\gamma_1, \gamma_2, \gamma_3\in \mathbb{K}:=\mathbb{Q}(\sqrt{5}),$ we can take $D:=2$. The left hand side of (\ref{eq:15a}) is not zero. The proof is done by contradiction. Suppose the contrary. Then
$$
2^a\sqrt{5}=\alpha^{n_1}+\alpha^{n_2}+\alpha^{n_3}+\alpha^{n_4}+\alpha^{n_5}.
$$
Conjugating the above relation in the field $\mathbb{Q}(\sqrt{5})$, we get
$$
-2^a\sqrt{5}=\beta^{n_1}+\beta^{n_2}+\beta^{n_3}+\beta^{n_4}+\beta^{n_5}.
$$
Combining the above two equations, we get
$$
\alpha^{n_1}< \alpha^{n_1}+\alpha^{n_2}+\alpha^{n_3}+\alpha^{n_4}=|\beta^{n_1}+\beta^{n_2}+\beta^{n_3}+\beta^{n_4}+\beta^{n_5}|\leqslant |\beta|^{n_1}+|\beta|^{n_2}+|\beta|^{n_3}+|\beta|^{n_4}+|\beta|^{n_5}< 1,
$$
and leads to contradiction since $n_1>400.$

\noindent As done before, here, we can take $A_1:=1.4, A_2:=0.5$ and $B:=n_1$. Let us estimate $h(\gamma_3)$.
We can see that, 
$$
\gamma_3=\dfrac{\sqrt{5}}{1+\alpha^{n_2-n_1}+\alpha^{n_3-n_1}+\alpha^{n_4-n_1}+\alpha^{n_5-n_1}}<\sqrt{5}
$$
and
$$
 \gamma_3^{-1}=\dfrac{1+\alpha^{n_2-n_1}+\alpha^{n_3-n_1}+\alpha^{n_4-n_1}+\alpha^{n_5-n_1}}{\sqrt{5}}<\sqrt{5}
$$
Hence $|\log \gamma_3|<1.$ Then 
\begin{align*}
h(\gamma_3)\leqslant &\log (5\sqrt{5})+|n_2-n_1|\left(\dfrac{\log\alpha}{2}\right)+|n_3-n_1|\left(\dfrac{\log\alpha}{2}\right)+|n_4-n_1|\left(\dfrac{\log\alpha}{2}\right)+|n_5-n_1|\left(\dfrac{\log\alpha}{2}\right)\\
=& \log(5\sqrt{5})+(n_1-n_2)\left(\dfrac{\log\alpha}{2}\right)+(n_1-n_3)\left(\dfrac{\log\alpha}{2}\right)+(n_1-n_4)\left(\dfrac{\log\alpha}{2}\right)+(n_1-n_5)\left(\dfrac{\log\alpha}{2}\right);
\end{align*}
so we can take
$$
A_3:= 5+(n_1-n_2)\log \alpha+(n_1-n_3)\log \alpha+ (n_1-n_4)\log \alpha+(n_1-n_5)\log \alpha.
$$
Then a lower bound on the left-hand side  of (\ref{eq:15a}) is 
$$
\exp (-c_4\cdot(1+\log n_1)\cdot 1.4\cdot 0.5\cdot(5+(n_1-n_2)\log \alpha+(n_1-n_3)\log \alpha+ (n_1-n_4)\log \alpha)),
$$
where $c_4=1.4\cdot 30^6\cdot 3^{4.5}\cdot 2^2\cdot (1+\log 2)<9.7\times 10^{11}.$

\noindent So, inequality (\ref{eq:15a}) yields 
\begin{equation}\label{eq:16a}
n_1\log \alpha< 1.4\times 10^{12}\log n_1\cdot (5+(n_1-n_2)\log \alpha+(n_1-n_3)\log \alpha+ (n_1-n_4)\log \alpha+(n_1-n_5)\log \alpha).
\end{equation}
Using now (\ref{eq:4}), (\ref{eq:9}) and (\ref{eq:12}) in the right-most terms of the above inequality (\ref{eq:16a}) and performing the respective calculation, and with the help of \textit{Mathematica},  we find that 
$$
 n_1 < 1.6\times 10^{73}.
$$

We record what we have proved.
\begin{lemma}
If $(n_1, n_2, n_3, n_4, n_5, a)$ is a positive solution of (\ref{eq1}) with $n_1\geqslant n_2\geqslant n_3\geqslant n_4\geqslant n_5$, then 
$$
a\leqslant  n_1 < 1.6\times 10^{73}.
$$
\end{lemma}

\section{Reduction the bound on $n$}
The goal of this section is to reduce the upper bound on $n$ to a size that can be handled. To do this, we shall use Lemma \ref{lem1} five times. Let us consider 
\begin{equation}\label{eq:15}
z_1:=a\log 2-n_1\log \alpha+ \log \sqrt{5}.
\end{equation}
From equation (\ref{eq:15}), (\ref{eq:2}) can be written as
\begin{equation}\label{eq:16}
\left|1-e^{z_1}\right|< \dfrac{11.5}{\alpha^{n_1-n_2}}.
\end{equation}
Associating (\ref{eq1}) and Binet's formula for the Fibonacci sequence, we have
$$
\dfrac{\alpha^{n_1}}{\sqrt{5}}=F_{n_1}+\dfrac{\beta^{n_1}}{\sqrt{5}}<   \displaystyle \sum _{k=1}^{5}F_{n_k}=2^a,
$$
hence 
$$
\dfrac{\alpha^{n_1}}{\sqrt{5}}< 2^a,
$$

\noindent which leads to $z_1>0.$ This result together with (\ref{eq:16}), gives
$$
0< z_1< e^{z_1}-1< \dfrac{11.5}{\alpha^{n_1-n_2}}.
$$
Replacing (\ref{eq:15}) in the inequality and dividing both sides of the resulting inequality by $\log \alpha$, we get
\begin{equation}\label{eq:17}
0< a \left(\dfrac{\log 2}{\log \alpha}\right)-n_1+ \left(\dfrac{\log \sqrt{5}}{\log \alpha}\right)<\dfrac{11.5}{\log \alpha}\cdot \alpha^{-(n_1-n_2)}< 24\cdot \alpha^{-(n_1-n_2)}.
\end{equation} 

\noindent We put 
$$
\tau:=\dfrac{\log 2}{\log \alpha}, \quad \mu:=\dfrac{\log \sqrt{5}}{\log \alpha},\quad A:=24,\quad \text{and}\quad B:=\alpha.
$$
$\tau$ is an irrational number. We also put $M:= 1.6\times 10^{73}$, which is an upper bound on $a$ by Lemma \ref{lem1} applied to inequality, that 
$$
n_1-n_2< \dfrac{\log(Aq/\varepsilon)}{\log B},
$$
where $q> 6M$ is a denominator of a convergent of the continued fraction of $\tau$ such that $\varepsilon:=\norm{\mu q}-M\norm{\tau q}>0.$ A computation with \textit{SageMath} revealed that if  $(n_1, n_2, n_3, n_4, n_5, a)$ is a possible solution of the equation (\ref{eq1}), then 

$$
n_1-n_2\in [0, 367].
$$

\noindent Let us now consider a second function, derived from (\ref{eq:5}) in order to find an improved upper bound on $n_1-n_2$.

\noindent Put 
$$
z_2:=a\log 2-n_1\log \alpha+ \log \Upsilon(n_1-n_2)
$$
where $\Upsilon$ is the function given by the formula $\Upsilon(t):=\sqrt{5}\left(1+\alpha^{-t}\right)^{-1}.$ From (\ref{eq:5}), we have 
\begin{equation}\label{eq:18}
\left| 1-e^{z_2}\right|< \dfrac{8}{\alpha^{n_1-n_3}}.
\end{equation}
Using (\ref{eq1}) and the Binet's formula for the Fibonacci sequence, we have
$$
\dfrac{\alpha^{n_1}}{\sqrt{5}}+\dfrac{\alpha^{n_2}}{\sqrt{5}}= F_{n_1}+F_{n_2}+\dfrac{\beta^{n_1}}{\sqrt{5}}+\dfrac{\beta^{n_2}}{\sqrt{5}}<F_{n_1}+F_{n_2}+1\leqslant F_{n_{1}}+F_{n_{2}}+F_{n_{3}}+F_{n_{4}}=2^a.
$$
Therefore $1< 2^a\alpha^{-n_1}\sqrt{5}\left(1+\alpha^{n_2-n_1}\right)^{-1}$ and so $z_2>0$. This with (\ref{eq:18}) give
$$
0<z_2\leqslant e^{z_2}-1< \dfrac{8}{\alpha^{n_1-n_3}}.
$$
Putting expression of $z_2$ in the above inequality and arguing as in (\ref{eq:17}), we obtain
\begin{equation}\label{eq:19}
0<a\left(\dfrac{\log 2}{\log \alpha}\right)-n_1+\dfrac{\log\Upsilon(n_1-n_2)}{\log \alpha}< 17\cdot \alpha^{-(n_1-n_3)}.
\end{equation}
As done before, we take again $M:=1.6\times 10^{73}$ which is the upper bound on $a$, and, as explained before, we apply Lemma \ref{lem1} to inequality (\ref{eq:19}) for all choices $n_1-n_2\in [0,367]$ except when $n_1-n_2=2,6$. With the help of \textit{SageMath}, we find that if $(n_1, n_2, n_3, n_4, n_5, a)$ is a possible solution of the equation (\ref{eq1}) with $n_1-n_2\neq 2$ and $n_1-n_2\neq 6$, then $n_1-n_3\in [0, 367]$.

\noindent Study of the cases $n_1-n_2\in \{2,6\}$. For these cases, when we apply Lemma \ref{lem1} to the expression (\ref{eq:19}), the corresponding parameter $\mu$ appearing in Lemma \ref{lem1} is 
$$
\dfrac{\log \Upsilon(t)}{\log\alpha}=
\left\{
    \begin{array}{lllll}
        1 & \text{if} & t=2;\\
        3-\dfrac{\log2}{\log\alpha} & \text{if} & t=6.
    \end{array}
\right.
$$
In both case, the parameters $\tau$ and $\mu$ are linearly dependent, which yield that the corresponding value of $\varepsilon$ from Lemma \ref{lem1} is always negative and therefore the reduction method is not useful  for reducing the bound on $n$ in these instances. For this, we need to treat these cases differently. 

\noindent However, we can see that if $t=2$ and $6$, then resulting inequality from (\ref{eq:19}) has the shape $0<|x\tau-y|<17\cdot \alpha^{-(n_1-n_3)}$ with $\tau$ being irrational number and $x,y\in \mathbb{Z}$. Then, using the known proprieties of the convergents of the continued fractions to obtain a nontrivial lower bound for $|x\tau-y|$. Let see how to do.
 
\noindent For $n_1-n_2=2$, from (\ref{eq:19}), we get that
\begin{equation}\label{eq:20}
0<a\tau -(n_1-1)< 17\cdot \alpha^{-(n_1-n_3)}, \quad \text{where}\quad \tau=\dfrac{\log 2}{\log \alpha}.
\end{equation}
Let $[a_1, a_2, a_3, a_4,\ldots]=[1,2,3,1,\ldots]$ be the continued fraction of $\tau$, and let denote $p_k/q_k$ it $k$th convergent. By Lemma \ref{lem2}, we know that $a<1.6\times 10^{73}$. An inspection in \textit{SageMath} reveals that 
$$
15177625911361796815938210846220393654109270183406839528998805820407221653
=q_{147}
$$
$$
<1.6\times 10^{73}<
$$
$$
q_{148}=75583009274523299909961213530369339183941874844471761873846700783141852920.
$$
Furthermore, $a_M:=\max \{a_i: i=0,1,\ldots,114\}=134.$ So, from the proprieties of continued fractions, we obtain that 
\begin{equation}\label{eq:21}
|a\tau-(n_1-1)|> \dfrac{1}{(a_M+2)a}.
\end{equation} 
Comparing (\ref{eq:20}) and (\ref{eq:21}), we get
$$
\alpha^{n_1-n_3}<17\cdot (134+2)a.
$$
Taking $\log$ on both sides of above equation and divide the obtained result by $\log\alpha$, we get 
$$
n_1-n_3<367.
$$

\noindent In order to avoid repetition, we freely omit the details for the case $n_1-n_2=6.$ Here, we get again $n_1-n_3<367$.

\noindent This completes the analysis of the two special cases $n_1-n_2=2$ and $n_1-n_2=6.$ Consequently $n_1-n_3\leqslant 367$ always holds.

\noindent Now let us use (\ref{eq:10}) in order to find improved upper bound on $n_1-n_4.$ Put 
$$
z_3:= a\log 2-n_1\log\alpha+\log\Upsilon_1(n_1-n_2, n_1-n_3),
$$ 
where $\Upsilon$ is the function given by the formula $\Upsilon_1(t,s):=\sqrt{5}\left(1+\alpha^{-t}+\alpha^{-s}\right)^{-1}.$ From (\ref{eq:10}), we have
\begin{equation}\label{eq:22}
|1-e^{z_3}|<\dfrac{4}{\alpha^{n_1-n_4}}.
\end{equation}

\noindent Note that, $z_3\neq 0$; thus, two cases arise: $z_3>0$ and $z_3< 0$.

\noindent If $z_3>0,$ then
$$
0<z_3\leqslant e^{z_3}-1<\dfrac{4}{\alpha^{n_1-n_4}}.
$$
Suppose now $z_3< 0.$ It is easy to check that $4/\alpha^{n_1-n_4}<1/2$ for all $n_1>400$ and $n_4\geqslant 2$. From (\ref{eq:22}), we have that 
$$
|1-e^{z_3}|<1/2 \quad \text{and therefore}\quad e^{|z_3|}<2.
$$
Since $z_3<0$, we have:
$$
0<|z_3|\leqslant e^{|z_3|}-1=e^{|z_3|}\left|e^{|z_3|}-1\right|< \dfrac{8}{\alpha^{n_1-n_4}}
$$
which gives 
$$
0<|z_3|<\dfrac{8}{\alpha^{n_1-n_4}}
$$
holds for $z_3<0$,  $z_3>0$ and for all for all $n_1>400$, and $n_4\geqslant 2$.
Replacing the expression of $z_3$ in the above inequality and arguing again as before, we conclude that
\begin{equation}\label{eq:23}
0< \left|a\left(\dfrac{\log 2}{\log \alpha}\right)-n_1+ \dfrac{\log \Upsilon_1(n_1-n_2, n_1-n_3)}{\log \alpha}\right|<17\cdot \alpha^{-(n_1-n_4)}.
\end{equation}

\noindent Here, we also take, $M:=1.6\times 10^{73}$ and we apply Lemma \ref{lem1} in inequality (\ref{eq:23}) for all choices $n_1-n_2\in \{0,367\}$ and $n_1-n_3\in \{0, 367\}$ except when 
$$
(n_1-n_2,n_1-n_3)\in \{(0,3), (1,1), (1,5), (3,0), (3,4), (4,3), (5,1), (7,8), (8,7)\}.
$$
Indeed, with the help of \textit{SageMath} we find that if $(n_1, n_2, n_3, n_4, n_5, a)$ is a possible solution of the equation (\ref{eq1}) excluding these cases presented before. Then $n_1-n_4\leqslant 367.$

\noindent \underline{SPECIAL CASES}. We deal with the cases when 
$$
(n_1-n_2,n_1-n_3)\in \{(1,1), (3,0), (4,3), (5,1), (8,7)\}.
$$
It is easy to check that

$$
\dfrac{\log \Upsilon_1(t,s)}{\log\alpha}=
\left\{
    \begin{array}{lllll}
        0, & \text{if} & (t,s)=(1,1);\\
        0, & \text{if} & (t,s)=(3,0);\\
        1, & \text{if} & (t,s)=(4,3);\\
        2-\dfrac{\log 2}{\log \alpha}, & \text{if} & (t,s)=(5,1);\\
        3-\dfrac{\log 2}{\log \alpha}, & \text{if} & (t,s)=(8,7).
    \end{array}
\right.
$$
As we explained before, when we apply Lemma \ref{lem1} to the expression (\ref{eq:23}), the parameters $\tau$ and $\mu$ are linearly dependent, so the corresponding value of $\varepsilon$ from Lemma \ref{lem1} is always negative in all cases. For this reason, we shall treat these cases differently.

\noindent Here, we have to solve the equations
$$
F_{n_2+1}+2F_{n_2}+F_{n_4}+F_{n_5}=2^a,\quad  2F_{n_2+3}+F_{n_2}+F_{n_4}+F_{n_5}=2^a,\quad F_{n_2+4}+F_{n_2}+F_{n_2+1}+F_{n_4}+F_{n_5}=2^a,
$$
\begin{equation}\label{eq:24}
 F_{n_2+5}+F_{n_2}+F_{n_2+4}+F_{n_4}+F_{n_5}=2^a,\quad \text{and}\quad F_{n_2+8}+F_{n_2}+F_{n_2+1}+F_{n_4}+F_{n_5}=2^a
\end{equation}
in positive integers $n_2, n_4, n_5$ and $a$. To do so, we recall the following well-known relation between the Fibonacci and the Lucas numbers:
\begin{equation}\label{eq:25}
L_k=F_{k-1}+F_{k+1} \quad \text{for all} \quad k\geqslant 1.
\end{equation}
From (\ref{eq:25}) and (\ref{eq:24}), we have the following identities 

$$
F_{n_2+1}+2F_{n_2}+F_{n_4}+F_{n_5}=F_{n_2+2}+F_{n_2}+F_{n_4}+F_{n_5}=F_{k+2}+F_{k}+F_{m}+F_{w}, 
$$
$$
2F_{n_2+3}+F_{n_2}+F_{n_4}+F_{n_5}=F_{n_2+2}+F_{n_2+4}+F_{n_4}+F_{n_5}=F_{k+2}+F_{k+4}+F_{m}+F_{w},
$$
\begin{equation}\label{eq:26}
 F_{n_2+4}+F_{n_2}+F_{n_2+1}+F_{n_4}+F_{n_5}=F_{n_2+2}+F_{n_2+4}+F_{n_4}+F_{n_5}=F_{k+2}+F_{k+4}+F_{m}+F_{w}, 
\end{equation}
$$
F_{n_2+5}+F_{n_2}+F_{n_2+4}+F_{n_4}+F_{n_5}=2F_{n_2+2}+2F_{n_2+4}+F_{n_4}+F_{n_5}=2F_{k+2}+2F_{k+4}+F_{m}+F_{w},
$$
$$
\text{and}\quad F_{n_2+8}+F_{n_2}+F_{n_2+1}+F_{n_4}+F_{n_5}= 2F_{n_2+6}+2F_{n_2+4}+F_{n_4}+F_{n_5}=2F_{k+6}+2F_{k+4}+F_{m}+F_{w},
$$
hold for all $k, m, w\geqslant 0$.

\noindent Equations (\ref{eq:24}) are transformed into the equations 
\begin{equation}\label{eq:27}
L_{k+1}+F_m+F_{w}=2^a,\quad L_{k+3}+F_{m}+F_{w}=2^a,\quad 2L_{k+3}+F_m+F_{w}=2^a,\quad 2L_{k+5}+F_m+F_{w}=2^a,
\end{equation}
to be resolved in positive integers $k, m, w$ and $a$.

\noindent A quick search in SageMath and analytical resolution leads to : 

$$
(k,m,w,a)\in \{(4, 4, 3, 4), (6, 3, 1, 5), (6, 3, 2, 5), (9, 4, 3, 7), (10, 10, 3, 8)\}\quad \text{for}\quad L_{k+1}+F_m+F_{w}=2^a,
$$

$$
(k,m,w,a)\in \{(4, 3, 1, 5), (4, 3, 2, 5), (7, 4, 3, 4)\} \quad \text{for}\quad L_{k+3}+F_{m}+F_{w}=2^a,
$$

$$
(k,m,w,a)\in \{(1, 1, 1, 4), (4,4,4,6), (7, 5, 5, 8),(7, 6, 3, 8)\} \quad \text{for}\quad 2L_{k+3}+F_m+F_{w}=2^a,
$$

$$
(k,m,w,a)=(5, 5, 5, 8) \quad \text{for}\quad 2L_{k+5}+F_m+F_{w}=2^a.
$$

\noindent We will nonetheless use $(t, s) \in \{(5,1),(8,7)\}$ for the purpose of showing that  Legendre's criterion is applicable when linear dependence arises in any subsequent cases. Consider
$$
\dfrac{\log \Upsilon_1(5, 1)}{\log \alpha}=2-\dfrac{\log 2}{\log \alpha}.
$$
We use the above expression to obtain the inequality
$$
0<\left|(a-1)\left(\dfrac{\log 2}{\log \alpha}\right)-(n_1- 2)\right|<17\cdot \alpha^{-(n_1-n_4)},
$$

\noindent to which Legendre's criterion may be applied.

\noindent Likewise when we consider 
$$
\dfrac{\log \Upsilon_1(8, 7)}{\log \alpha}=3-\dfrac{\log 2}{\log \alpha},
$$
and we use the above expression to obtain the inequality
$$
0< \left|(a-1)\left(\dfrac{\log 2}{\log \alpha}\right)-(n_1- 3)\right|<17\cdot \alpha^{-(n_1-n_4)},
$$
to which Legendre's criterion may be applied again. In both cases, $n_1-n_4< 367.$

\noindent This completes the analysis of special cases.

\noindent Now, let us use (\ref{eq:13}) in order to find improved upper bound on $n_1-n_5.$ Put
$$
z_4:=a\log 2-n_1\log \alpha+ \log\Upsilon_2(n_1-n_2, n_1-n_3, n_1-n_4),
$$ 
where $\Upsilon_2$ is the function given by the formula 
$$
\Upsilon_2(t,u,v):= \sqrt{5}\left( 1+ \alpha^{-t}+\alpha^{-u}+\alpha^{-v}\right)^{-1}
$$
with $t=n_1-n_2, u=n_1-n_3$ and $v=n_1-n_4.$ From (\ref{eq:13}), we get 
\begin{equation}\label{eq:28}
|1-e^{z_4}|<\dfrac{2}{\alpha^{n_1-n_5}}.
\end{equation}

\noindent Since $z_4\neq 0$, as before, two cases arise: $z_4< 0$ and $z_4> 0$.

\noindent If $z_4>0$, then 
$$
0<z_4\leqslant e^{z_4}-1<\dfrac{2}{\alpha^{n_1-n_5}}.
$$

\noindent Suppose now that $z_4<0$. We have $2/\alpha^{n_1-n_5}<1/2$ for all $n_1>400$ and $n_5\geqslant 5$. 

\noindent Then, from (\ref{eq:28}), we have 
$$
|1-e^{z_4}|<\dfrac{1}{2}
$$
and therefore $e^{|z_4|}<2$.

\noindent Since $z_4<0$, we have : 
$$
0<|z_4|\leqslant e^{|z_4|}-1=e^{|z_4|}\left|e^{|z_4|}-1\right|< \dfrac{4}{\alpha^{n_1-n_5}}
$$
which gives 
$$
0<|z_4|<\dfrac{4}{\alpha^{n_1-n_5}}
$$
for the both cases ($z_4< 0$ and $z_4>0$ ) and holds for all $n_1>400.$

\noindent Replacing the expression of $z_4$ in the above inequality and arguing again as before, we conclude that
\begin{equation}\label{eq:29}
0< \left|a\left(\dfrac{\log 2}{\log \alpha}\right)-n_1+ \dfrac{\log \Upsilon_2(n_1-n_2, n_1-n_3, n_1-n_4)}{\log \alpha}\right|<9\cdot \alpha^{-(n_1-n_5)}.
\end{equation}

\noindent Here, we also take, $M:=1.6\times 10^{73}$ and we apply Lemma \ref{lem1} one more time in inequality (\ref{eq:29}) for all choices $n_1-n_2\in \{0,367\}$, $n_1-n_3\in \{0, 367\}$ and $n_1-n_4 \in \{0,367\}$ with $(n_1, n_2, n_3, n_4, n_5, a)$ a possible solution of equation (\ref{eq1}), and by omitting the  study of special cases (because it gives a solution presented in Theorem \ref{theo4} ), we get: 

$$n_1-n_5< 369.$$

\noindent Finally let us use (\ref{eq:15a}) in order to find improved upper bound on $n_1.$ Put
$$
z_5:=a\log 2-n_1\log \alpha+ \log\Upsilon_3(n_1-n_2, n_1-n_3, n_1-n_4, n_1-n-n_5),
$$ 
where $\Upsilon_3$ is the function given by the formula 
$$
\Upsilon_3(t,u,v,y):= \sqrt{5}\left( 1+ \alpha^{-t}+\alpha^{-u}+\alpha^{-v}+\alpha^{-y}\right)^{-1}
$$
with $t=n_1-n_2, u=n_1-n_3$, $v=n_1-n_4$ and $y=n_1-n_5.$

\noindent From (\ref{eq:15a}), we get 
\begin{equation}\label{eq:30}
|1-e^{z_5}|<\dfrac{3}{\alpha^{n_1}}.
\end{equation}

\noindent Since $z_5\neq 0$, as before, two cases arise: $z_5< 0$ and $z_5> 0$.

\noindent If $z_5>0$, then 
$$
0<z_5\leqslant e^{z_5}-1<\dfrac{3}{\alpha^{n_1}}.
$$

\noindent Suppose now that $z_5<0$. We have $3/\alpha^{n_1}<1/2$ for all $n_1>400$. Then, from (\ref{eq:30}), we have 
$$
|1-e^{z_5}|<\dfrac{1}{2}
$$
and therefore $e^{|z_5|}<2$.

\noindent Since $z_5<0$, we have : 
$$
0<|z_5|\leqslant e^{|z_5|}-1=e^{|z_5|}\left|e^{|z_5|}-1\right|< \dfrac{6}{\alpha^{n_1}}
$$
which gives 
$$
0<|z_5|<\dfrac{6}{\alpha^{n_1}}
$$
for the both cases ($z_5< 0$ and $z_5>0$ ) and holds for all $n_1>400.$

\noindent Replacing the expression of $z_5$ in the above inequality and arguing again as before, we conclude that
\begin{equation}\label{eq:31}
0< \left|a\left(\dfrac{\log 2}{\log \alpha}\right)-n_1+ \dfrac{\log \Upsilon_2(n_1-n_2, n_1-n_3, n_1-n_4,n_1-n_5)}{\log \alpha}\right|<13\cdot \alpha^{-n_1}.
\end{equation}

\noindent Here, we also take, $M:=1.6\times 10^{73}$ and we apply Lemma \ref{lem1} last time in inequality (\ref{eq:31}) for all choices $n_1-n_2\in \{0,367\}$, $n_1-n_3\in \{0, 367\}$, $n_1-n_4 \in \{0,367\}$ and $n_1-n_5 \in \{0,369\}$ with $(n_1, n_2, n_3, n_4, n_5, a)$ a possible solution of equation (\ref{eq1}), and by omitting the study of special cases (because it gives a solution presented in Theorem \ref{theo4}, we get: 

$$n_1< 369.$$
This is false because this contradicts our initial hypothesis $n_1 > 400$.

\noindent This ends the proof of our main theorem.

\end{proof}

\newpage
\section{Table}\label{5}
In this table, we have the complete list of all non-trivial solutions of the Diophantine equation (\ref{eq1}). We have one hundred and six non-trival solutions.
\begin{tiny}
\begin{multicols}{3}
\begin{center}
    \begin{tabular}{|c|c|c|c|c|c|}
\hline
Numbering & $n_5$&$n_4$&$n_3$&$n_2$&$n_1$\\ \hline 
1&1& 1& 1& 3& 4\\ \hline 
2& 1& 1& 1& 5& 6\\ \hline 
3& 1& 1& 1& 6& 8 \\ \hline
4& 1& 1& 1& 9& 16\\ \hline
5& 1& 1& 2& 3& 4\\ \hline
6& 1& 1& 2& 5& 6 \\ \hline
7& 1& 1& 2& 6& 8\\ \hline
8& 1& 1& 2& 9& 16\\ \hline
9& 1& 1& 3& 3& 3\\ \hline
10& 1& 1& 3& 5& 10\\ \hline 
11& 1& 1& 4& 4& 6 \\ \hline
12& 1& 1& 4& 9& 11 \\ \hline
13& 1& 1& 6& 7& 13 \\ \hline
14& 1& 1& 8& 11& 12 \\ \hline
15& 1& 1& 10& 10& 12\\ \hline
16& 1& 2& 2& 3& 4 \\ \hline
17& 1& 2& 2& 5& 6 \\ \hline
18& 1& 2& 2& 6& 8\\ \hline
19& 1& 2& 2& 9& 16 \\ \hline
20& 1& 2& 3& 3& 3 \\ \hline
21& 1& 2& 3& 5& 10 \\ \hline
22& 1& 2& 4& 4& 6\\ \hline
23& 1& 2& 4& 9& 11 \\ \hline
24& 1& 2& 6& 7& 13 \\ \hline
25& 1& 2& 8& 11& 12 \\ \hline
26& 1& 2& 10& 10& 12 \\ \hline
27& 1& 3& 3& 4& 6 \\ \hline
28& 1& 3& 3& 9& 11 \\ \hline
29& 1& 3& 4& 4& 10 \\ \hline
30& 1& 3& 4& 5& 5 \\ \hline
31& 1& 3& 4& 5& 8 \\ \hline
32& 1& 3& 4& 7& 7 \\ \hline
33& 1& 3& 6& 6& 7 \\ \hline
34& 1& 3& 7& 8& 16 \\ \hline
35& 1& 3& 9& 14& 15 \\ \hline
 \end{tabular}
\end{center} 

\columnbreak
\begin{center}
    \begin{tabular}{|c|c|c|c|c|c|}
\hline
Numbering & $n_5$&$n_4$&$n_3$&$n_2$&$n_1$\\ \hline 
36& 1& 4& 5& 8& 9\\ \hline
37& 1& 4& 7& 7& 9\\ \hline
38& 1& 5& 5& 6& 7\\ \hline
39& 1& 6& 6& 7& 9\\ \hline
40& 1& 6& 7& 8& 8\\ \hline
41& 2& 2& 2& 3& 4\\ \hline
42& 2& 2& 2& 5& 6\\ \hline
43& 2& 2& 2& 6& 8\\ \hline
44& 2& 2& 2& 9& 16\\ \hline
45& 2& 2& 3& 3& 3\\ \hline
46& 2& 2& 3& 5& 10\\ \hline
47& 2& 2& 4& 4& 6\\ \hline
48& 2& 2& 4& 9& 11\\ \hline
49& 2& 2& 6& 7& 13\\ \hline
50& 2& 2& 8& 11& 12\\ \hline
51& 2& 2& 10& 10& 12\\ \hline
52& 2& 3& 3& 4& 6\\ \hline
53& 2& 3& 3& 9& 11\\ \hline
54& 2& 3& 4& 4& 10\\ \hline
55& 2& 3& 4& 5& 5\\ \hline
56& 2& 3& 4& 5& 8\\ \hline
57& 2& 3& 4& 7& 7\\ \hline
58& 2& 3& 6& 6& 7\\ \hline
59& 2& 3& 7& 8& 16\\ \hline
60& 2& 3& 9& 14& 15\\ \hline
61& 2& 4& 5& 8& 9\\ \hline
62& 2& 4& 7& 7& 9\\ \hline
63& 2& 5& 5& 6& 7\\ \hline
64& 2& 6& 6& 7& 9\\ \hline
65& 2& 6& 7& 8& 8\\ \hline
66& 3& 3& 3& 3& 6\\ \hline
67& 3& 3& 3& 4& 10\\ \hline
68& 3& 3& 3& 5& 5\\ \hline
69& 3& 3& 3& 5& 8\\ \hline
70& 3& 3& 3& 7& 7\\ \hline

 \end{tabular}
\end{center} 

\columnbreak
\begin{center}
\begin{tabular}{|c|c|c|c|c|c|}
    \hline
Numbering & $n_5$&$n_4$&$n_3$&$n_2$&$n_1$\\ \hline
71& 3& 3& 5& 8& 9\\ \hline
72& 3& 3& 7& 7& 9\\ \hline
73& 3& 4& 4& 4& 5\\ \hline
74& 3& 4& 4& 4& 8\\ \hline
75& 3& 4& 5& 7& 13\\ \hline
76& 3& 4& 7& 8& 11\\ \hline
77& 3& 4& 7& 10& 10\\ \hline
78& 3& 4& 9& 9& 10\\ \hline
79& 3& 5& 6& 6& 13\\ \hline
80& 3& 6& 6& 8& 11\\ \hline
81& 3& 6& 6& 10& 10\\ \hline
82& 3& 6& 7& 11& 12\\ \hline
83& 3& 8& 9& 10& 12\\ \hline
84& 3& 8& 10& 11& 11\\ \hline
85& 3& 10& 10& 10& 11\\ \hline
86& 4& 4& 4& 8& 9\\ \hline
87& 4& 4& 5& 6& 7\\ \hline
88& 4& 5& 6& 6& 6\\ \hline
89& 4& 5& 6& 8& 16\\ \hline
90& 4& 6& 7& 7& 16\\ \hline
91& 4& 7& 8& 14& 15\\ \hline
92& 4& 9& 12& 13& 15\\ \hline
93& 4& 9& 13& 14& 14\\ \hline
94& 4& 11& 11& 13& 15\\ \hline
95& 5& 5& 5& 6& 13\\ \hline
96& 5& 5& 6& 8& 11\\ \hline
97& 5& 5& 6& 10& 10\\ \hline
98& 5& 5& 7& 11& 12\\ \hline
99& 5& 6& 7& 7& 11\\ \hline
100& 5& 7& 8& 9& 10\\ \hline
101& 5& 8& 9& 9& 9\\ \hline
102& 6& 6& 6& 7& 16\\ \hline
103& 6& 6& 8& 14& 15\\ \hline
104& 6& 10& 14& 16& 20\\ \hline
105& 7& 7& 7& 9& 10\\ \hline
106& 7& 7& 9& 9& 9\\ \hline
    \end{tabular}

\end{center}
\end{multicols}
\end{tiny}

\section{Open problem}
The sequence $4, 9, 15, 60, 106,\cdots$  counting the number of solutions of the equation $\displaystyle \sum_{k=1}^N F_{n_k} = 2^a$ for $N = 1, 2, 3,\cdots$ does not appear in Sloan's OEIS. The problem is: can one provide a general theory for these equations for any $N$? We created a new entry in OEIS for the  sequence $4, 9, 15, 60, 106,\cdots$. It can be found here {\color{blue}\url{https://oeis.org/A356928}}. 
\section*{Acknowledgments}

The authors would like to thank the referee for careful reading.

\vspace{1cm}
\textbf{Acknowledgment.} We thank the referee for a careful reading of the paper and for comments and suggestions, which improved its quality.

\pagebreak

\vspace{2cm}

 \end{document}